\documentclass[11pt]{amsart}
\usepackage{amssymb,hyperref}

\usepackage{graphicx}
\usepackage[matrix,arrow]{xy}
\newcommand{\PP}{\mathbb P}

\newcommand{\CC}{{\mathbb C}}
\newcommand{\NN}{{\mathbb N}}
\newcommand{\ZZ}{{\mathbb Z}}
\newcommand{\QQ}{{\mathbb Q}}

\newcommand{\KK}{{k}}
\newcommand{\NS}{\mathrm{NS}}
\newcommand{\MWL}{\mathrm{MWL}}
\newcommand{\Triv}{\mathrm{Triv}}

\newcommand{\Pic}{\mathrm{Pic}}

\newcommand{\tD}{2h}
\newcommand{\jD}{h}

\numberwithin{equation}{section}
\begin{document}

\title{Miyaoka's bound for conics on K3 surfaces and beyond}

\author{S\l awomir Rams}
\address{Institute of Mathematics, Jagiellonian University, 
ul. {\L}ojasiewicza 6,  30-348 Krak\'ow, Poland}
\email{slawomir.rams@uj.edu.pl}

\author{Matthias Sch\"utt}
\address{Institut f\"ur Algebraische Geometrie, Leibniz Universit\"at
  Hannover, Welfengarten 1, 30167 Hannover, Germany}

    \address{Riemann Center for Geometry and Physics, Leibniz Universit\"at
  Hannover, Appelstrasse 2, 30167 Hannover, Germany}

\email{schuett@math.uni-hannover.de}

\thanks{%
This research (S{\l}awomir Rams) was funded in part by
 the National Science Centre, Poland, Opus  grant
	no.\ 2024/53/B/ST1/01413.
The second author's research is conducted in the framework of the research training
group GRK 2965: From Geometry to Numbers,
funded by DFG}

\dedicatory{ Dedicated to Tetsuji Shioda on the occasion of his 85th birthday}


\begin{abstract}
We show that Miyaoka's bound for the number of conics on a degree-$2h$ 
K3 surface is attained for  high $h$,
and analogously for higher even degree (smooth) rational curves.
\end{abstract}

\maketitle

\newcommand{\XXd}{X_{d}}
\newcommand{\XXf}{X_{4}}
\newcommand{\XXp}{X_{5}}
\newcommand{\mF}{\mathcal F}
\newcommand{\MW}{\mathop{\mathrm{MW}}}
\newcommand{\mL}{\mathcal L}
\newcommand{\mR}{\mathcal R}
\newcommand{\Ruledeight}{S_{11}}
\newcommand{\Ruledfour}{S_{4}}
\newcommand{\DivisorRest}{\mathfrak Rest}
\newcommand{\Pl}{\Pi}
\newcommand{\reg}{\operatorname{reg}}

\newcommand{\IK}{{{\rm I}}}
\newcommand{\II}{{\mathop{\rm II}}}
\newcommand{\III}{{\mathop{\rm III}}}
\newcommand{\IV}{{\mathop{\rm IV}}}
\newcommand{\rr}{\mathfrak r}

\theoremstyle{remark}
\newtheorem{obs}{Observation}[section]
\newtheorem{rem}[obs]{Remark}
\newtheorem{example}[obs]{Example}
\newtheorem{ex}[obs]{Example}
\newtheorem{conv}[obs]{Convention}
\theoremstyle{definition}
\newtheorem{Definition}[obs]{Definition}
\theoremstyle{plain}
\newtheorem{prop}[obs]{Proposition}
\newtheorem{theo}[obs]{Theorem}
\newtheorem{Theorem}[obs]{Theorem}
\newtheorem{lemm}[obs]{Lemma}
\newtheorem{crit}[obs]{Criterion}
\newtheorem{claim}[obs]{Claim}
\newtheorem{Fact}[obs]{Fact}
\newtheorem{cor}[obs]{Corollary}
\newtheorem{assumption}[obs]{Assumption}
\newtheorem{conclusion}[obs]{Conclusion}
\newtheorem{summary}[obs]{Summary}

\newcommand{\ux}{\underline{x}}
\newcommand{\ud}{\underline{d}}
\newcommand{\ue}{\underline{e}}
\newcommand{\mmS}{{\mathcal S}}
\newcommand{\mmP}{{\mathcal P}}
\newcommand{\nlines}{\mbox{\texttt l}(\XXp)}
\newcommand{\ii}{\operatorname{i}}

\newcommand{\nonlinflec}{{\mathcal D}}
\newcommand{\linflec}{{\mathcal L}}
\newcommand{\flec}{{\mathcal F}}

\newcommand{\kfield}{{\KK}}


\section{Introduction}

Rational curves on algebraic surfaces
form a classical topic that has sparked great interest in the realm of enumerative geometry. Although configurations of such curves (especially lines and conics)
on various classes of surfaces have been studied for over 150 years, it was only
within the last decade that most  questions concerning line configurations on polarized  K3-surfaces over various fields have been answered (see  e.g. \cite{RS-64}, \cite{DIS}, \cite{Shimada56}, \cite{RS-112}, \cite{Miyaoka}, \cite{degt}, \cite{degt-tritangents} \cite{DR-2023}). Far less is known about the behaviour of large configurations of conics 
(see e.g. \cite{bauer-skew}, \cite{Miyaoka}, \cite{degt-800}, \cite{sarti-m20}, \cite{degt-nagoya}).
 
For a fixed  algebraically  closed field $\kfield$,  
integers  $d \geq 1$ and $h \geq 2$ let us define 
$$
\rr_{\kfield}(h,d) := \max \{ r_d(X) :  X \mbox{ is a 
	degree-}2h \mbox{ smooth K3 surface over } \kfield\}
$$
where 
 $r_d(X)$ is the number of degree-$d$ rational curves on the 
  surface $X$
 (see \cite{Miyaoka}). 
 As a consequence of the orbibundle Miyaoka-Yau-Sakai inequality \cite{Miyaoka-orbi} Miyaoka obtained (see \cite[Proposition A]{Miyaoka}) the following inequality 
\begin{equation} \label{eq-miyaoka}
 \rr_{\CC}(h,d) \leq \frac{24 h}{h - 2 d^2} \quad \mbox{  for  } h > 2 d^2 \, ,
 \end{equation}
and asked  to what extent this bound is sharp  (see \cite[p.~920]{Miyaoka}). 

Obviously, \eqref{eq-miyaoka} implies the inequality $ \rr_{\CC}(h,d)  \leq 24$  for $h > 50 d^2$. 
Focussing on lines, i.e.\ $d=1$, 
 the function $h \mapsto \rr_{\CC}(h,1)$ becomes periodic for large $h$ by 
\cite[Theorem~1.5]{degt},  and in fact,
   $$
   \rr_{\CC}(h,1) \in \{21, 22, 24\} \quad \mbox{for } h \gg 0.
   $$
(For precise arithmetic conditions  on the degree $2h$ of the polarization
for Miyaoka's bound of 24 lines to be attained, see \cite[Theorem~1.5]{degt}.) Consequently, 
\eqref{eq-miyaoka} fails to be sharp for lines on complex degree-$2h$ K3-surfaces for infinitely many $h$. On the other hand, by \cite[Theorem~1.1]{RS-24},
we have
$$
 \rr_{\kfield}(h,d) = 24 \quad \mbox{ for }   d \geq 3, \, \,  h > 42d^2 \mbox{ and } \mbox{char}( \KK) \neq 2,3
$$
(where the bound even holds if we consider all rational curves of degree $\leq d$ altogether).
Unfortunately, the main construction from \cite{RS-24} yields  K3 surfaces  with nodal rational curves, so it cannot be applied for 
conics. 
The  aim of this paper is to prove that
there are no arithmetic conditions on the degree  of the polarization
for Miyaoka's bound of at most 24 rational curves on degree-$2h$ K3 surfaces to be attained
by even degree {\sl smooth} rational curves (once $h$ is large enough), 
as captured in our main theorem:
\begin{theo}
\label{thm}
Let $d, h\in\ZZ_{>1}$ and let $k$ be an algebraically closed field of characteristic $p\geq 0$, $p\neq 2,3$. 
There is a degree-$2h$ K3 surface $X$ over $k$ containing
 24 smooth degree-$d$ rational curves if 
\begin{enumerate}
\item
$h\geq 2d(2d+3)-2$ and
\item
$d$ or $h$ is even. 
\end{enumerate}
\end{theo}
Note that both for even and odd integer $d$ Theorem~\ref{thm} implies that
Miyaoka's bound of at most 24 rational curves of a given degree $d$ 
on a smooth degree-$2h$ K3 surface
is attained by configurations of smooth rational degree-$d$ curves for infinitely many values of $h$.
For brevity, we only make this explicit for conics:

\begin{cor}
	\label{cor} Let  $h > 168$ and let $k$ be an algebraically closed field of characteristic $p\geq 0$, $p\neq 2,3$. 
	Then the maximal number of conics on a smooth degree-$2h$ K3 surface over $k$ is $24$ (i.e. $\rr_{\kfield}(h,2) = 24$).  \\
	In particular 
	Miyaoka's bound \eqref{eq-miyaoka} for the number of conics on a degree-$2h$ complex K3 surface is sharp for any $h>200$.
\end{cor}

\noindent
Indeed,  Theorem~\ref{thm} 
combined with \cite[Theorem~1.1.1]{RS-24} yields the equalities
$$
\rr_{\kfield}(h,2) = 24 \quad \mbox{ for } \, \,  h > 168 \mbox{ and } \mbox{char}( \KK) \neq 2,3 \, ,
$$
which completes the picture of the behaviour of large configurations of low-degree rational curves on high-degree K3 surfaces (at least in complex case - recall that \cite{degt} is devoted to lines on {\sl complex} K3 surfaces).

It is natural to ask whether the assumption that $h$ is even in 
Theorem~\ref{thm}.(2) can be weakened. We discuss briefly this question in 
$\S$.\ref{sec-finalremarks}, but avoid lengthy arguments in order to keep this note reasonably compact and independent of
 large computer-aided arguments. Finally, let us recall that 
\eqref{eq-miyaoka} fails over fields of characteristic $p=2,3$ (see \ref{ss:2,3} or \cite[Theorems~1.2, 1.3]{RS-24}).

\begin{conv}
We assume that the base field $k$ of characteristic $p\geq 0$
is algebraically closed.
Throughout this paper, all root lattices 
are assumed to be negative-definite,
and rational curves  are assumed to be irreducible.
\end{conv}

\section{Preliminaries}
\label{s:smooth}

As in \cite{Miyaoka}, \cite{degt}, \cite{RS-24} 
we consider pairs $(X,H)$ where $X$ is a smooth K3 surface  
over an algebraically closed field $k$ of characteristic $p$, and  $H$ 
is a very ample divisor with $H^2=\tD$. Obviously
the linear system $|H|$ defines an embedding
\[
X \hookrightarrow \PP^{\jD+1}
\]
Since its image is a smooth degree-$2h$ 
surface, we call  $(X,H)$  a polarized  degree-$2h$ K3 surface
(and simply write $X$ instead of  $(X,H)$ whenever it leads to no ambiguity).
In order to prove that a given  $H$ is very ample  we will use the following criterion from \cite{S-D}:

\begin{crit}
	\label{crit}
	Let $p \neq 2$ and let  $H$ be 
	a divisor on a K3 surface $X$. 
	If
	\begin{enumerate}
		\item
		$H.C > 0$ for every curve $C\subset X$; 
		\item
		$H.E>2$ for every irreducible curve $E\subset X$ of arithmetic genus $1$;
		\item
		$H^2\geq 4$, and for $H^2=8$, the divisor  
		$H$ is not $2$-divisible in $\Pic(X)$,
	\end{enumerate}
	then $H$ is very ample. 
\end{crit}

\section{Complex case, even $d$}
\label{s:even_d}

Here we prove Theorem~\ref{thm} for even $d$ and $k = \CC$. The restriction on the base field  plays an important role in the second step of the proof (c.f. $\S$~\ref{sec-step2}).

\begin{prop}
\label{thm:even_d_over_C}
For even $d$ and every   $h\geq 2d(2d+3)-2$
there exists a complex projective degree-$2h$ $K3$-surface with exactly $24$ smooth degree-$d$ rational curves. 
\end{prop}

\begin{proof}
We proceed in several steps.

\subsection{Step 1} 

We start by following closely the construction from \cite[\S 11.3]{RS-24}.
Assume that char$(k)\neq 2$ and consider squarefree polynomials $f,g\in k[t]$
of degree $4$ which are relatively prime such that $f-g$ is also squarefree of the same degree.
(One of the three polynomials could also be allowed to have degree $3$,
accounting for a zero at $\infty$.)
Then the extended Weierstrass form 
\[
y^2 = x(x-f)(x-g)
\]
defines an elliptic K3 surface $X$ over $\PP^1_t$
with twelve singular fibres of type $\IK_2$ at the zeroes of $f, g$ and $f-g$.
Generically, one has $\MW(X)\cong(\ZZ/2\ZZ)^2$
with disjoint sections 
\[
P_1 = (0,0), \;\;\; P_2 =  (0,f),\;\;\; P_3 =  (0,g)
\]
and $O$ the point at $\infty$.

Using Criterion \ref{crit}, one finds:

\begin{Fact} \cite[Fact 11.3]{RS-24}
Assume that $d$ is even.
Let $F$ denote a fibre and $N>d$. Then
$H=NF+\frac{d}{2} (O+P_1+P_2+P_3)$ is very ample.
\end{Fact}

We put $\Theta_{0,j} + \Theta_{1,j}$ for $j=1, \ldots, 12$ to denote the twelve  singular fibers of the type  $\IK_2$. Then
since $(O+P_1+P_2+P_3).\Theta_{i,j} = 2$ for every $i,j$ we have
$$
H.\Theta_{i,j} = d \;\; \mbox{ for every } \; i,j
$$

Therefore, any general choice of $X$ with non-degenerate singular fibre configuration admits
 degree-$\tD$ models for $h=d(2N-d)$
 with 24 smooth  degree-$d$ rational curves
in characteristic $\neq 2$.
 Since $d$ is even, we derive the congruence 
 \begin{equation} \label{eq-modular4d}
 H^2\equiv 2d^2 \equiv 0 \mod 4d.
 \end{equation}
 In particular, we see that there are $2d-1$ residue classes $H^2$ modulo $4d$ left to be covered.

\subsection{Step 2}  \label{sec-step2}

In this subsection, we restrict to  the complex numbers,
for otherwise the Noether--Lefschetz loci in consideration need not be non-empty,
but we will see in Section \ref{s:p>0}
how to overcome this in positive characteristic.

Consider the transcendental lattice $T(X)$
of a very general choice of $X$. 
By assumption, this has signature $(2,6)$, and in fact, it takes the shape
\begin{eqnarray}
\label{eq:T(X)}
T(X)
\cong
U(2)^2 \oplus A_1^4.
\end{eqnarray}
To see this, note that $\NS(X)$ is an index 4  overlattice of the trivial lattice $\Triv(X)=U\oplus A_1^{12}$
generated by zero section $O$ and fibre components.
As this is 2-elementary, so is $\NS(X)$ and thus also $T(X)$.
But then the claim follows from \cite{nikulin} by the  elementary verification 
that the discriminant form assumes non-integer values.

In the sequel, we shall use the fact that $T(X)$ represents every even integer by \eqref{eq:T(X)}.
This implies that, for any even $r<0$, we can choose a primitive vector $D \in T$ with $D^2 =r$. 
By  moduli theory for $K3$ surfaces there exists a K3 surface $W$
(deforming in a $5$-dimensional family)
 such that for its transcendental lattice we have
$$
T(W) \cong D^{\perp_{T(X)}}.
$$  
By definition, $W$ inherits the structure of elliptic surface from $X$.
In detail, we have the primitive closure
\[
\NS(W) = \overline{\NS(X) \oplus \ZZ D} \subset H^2(W,\ZZ),
\]
exhibiting $\NS(W)$ as an index 2 overlattice because $T(X)$ is 2-divisible as an integral lattice (not even!),
so any primitive vector $v\in T(X)$ has $v/2\in T(X)^\vee$ which glues to some element in $\NS(X)^\vee$.
The embedding of the hyperbolic plane $U\hookrightarrow \NS(X)$ corresponding to 
our fixed elliptic fibration on $X$ induces the claimed fibration (with section) on $W$.
Its singular fibres are encoded in  the roots of $\overline{A_1^{12} \oplus \ZZ D}$.
One directly verifies that this root lattice continues to be $A_1^{12}$ unless $D^2=-2$ 
(where $A_1^3\oplus \ZZ D$ may have $D_4$ as an overlattice)
or $D^2=-4$ (where  $A_1^2\oplus \ZZ D$ may have $A_3$ as an overlattice).
Under the condition that $D^2<-4$,
we thus obtain that all singular fibres of the fibration on $W$ have type $\IK_2$.
Hence
$$
\Triv(W) = \Triv(X) \;\;
\; \text{ and } \;\;\;
D \in \Triv(W)^{\perp}\subset\NS(W).
$$
It follows from the theory of Mordell--Weil lattices \cite{MWL} that there is a section $Q\in\MW(W)$ 
meeting each fibre in the identity component (met by $O$) such that
$D = \varphi(Q)$  where  	$\varphi$ is the orthogonal projection from $\NS(W)$ with respect to $\Triv(W)$.
Note that, in general, this projection requires tensoring with $\QQ$,
as it involves rational correction terms at the singular fibres,
but since $D$ is supported on the integral orthogonal complement of $\Triv(W)$,
there is presently no need for correction terms.
In other words, $Q$ lies in the narrow Mordell--Weil lattice MWL$^0(W)$ which will be instrumental in the sequel,
see Section \ref{s:p>0}.
In detail,  
$$	
 D = Q - O - ((Q.O)+2)F, 
$$
but then 
\begin{eqnarray}
\label{eq:D-Q}
-D^2 = 4 + 2 (Q.O) \;\;\; \text{ so } \;\;\; (Q.O) = -D^2/2-2.
\end{eqnarray}

\subsection{Step 3}

Choose $-D^2\in\{6,8,\hdots,4d+4\}$ to cover all even residue classes modulo $4d$
while preserving the configuration of singular fibres.
(Here $-D^2=4d$ may be omitted in view of \eqref{eq-modular4d}.)
Then \eqref{eq:D-Q} implies that
$$
(Q.O) \leq 2d.
$$
Since torsion section are always disjoint, we also get the following estimate
$$
(Q.P_i) = (Q.O)+2 = -D^2/2 \leq 2d+2 \;\;\; (i=1,2,3)
$$
from the zero height pairing
$$
0 =
\langle Q,P_i\rangle = 2 + (Q.O) + (P_i.O) - (Q.P_i).
$$


\subsection{Step 4} 

We put $H' := H + D$ and check when this defines a very ample divisor using Criterion \ref{crit}. 
We have
$$
(H')^2 = H^2 + D^2 = 4Nd + D^2 \geq 4d(N-1)-4
$$
and spell out
$$
H'= (N-(Q.O)-2) F +\frac{d}{2}(P_1+P_2+P_3) + \left(\frac{d}{2}-1\right)O  + Q. 
$$
Again, we check the conditions of Criterion \ref{crit}:
\begin{enumerate}
	\item
	$H'.F = H.F=2d\geq 4$;
		\item
	$H'.\Theta = H.\Theta = d\geq 2$ for any component $\Theta$ of a reducible fibre;
	\item 
	$H'.O = N-d$;
	
	\item
	$H'.P_i = N-(Q.O)-2 -d + (Q.P_i)  = N-d$ for each $i=1,2,3$;
	\item
	$H.Q = N +(2d-2)(Q.O) + 3d-4 \geq N+2$;
	\item
	$H'.P'\geq N-(Q.O)-2 \geq N-2d-2$ 
	for any other section $P'$;	
	\item
	$H'.D'\geq d'(N-(Q.O)-2) \geq 2(N-2d-2)$ 
	for any irreducible multisection $D'$ of degree $d'>1$;
in particular, if $p_a(D')=1$, then $H'.D'>2$ as required per Criterion \ref{crit} if $N>2d+3$.
\end{enumerate}

By varying $N>2d+3$ and taking $D$ with $D^2$ in the above range,
we thus derive the following conclusion using Criterion \ref{crit}:

\begin{conclusion}
For any even $d$,
we obtain projective models of $W$ 
for any degree  $2h\geq 4d(2d+3)-4$
with 24 smooth rational curves of degree $d$.
\end{conclusion}

In particular, this completes the proof of 
Proposition \ref{thm:even_d_over_C}. 
\end{proof}

\begin{rem}
This already proves that Miyaoka's bound for conics  (\cite[Proposition A]{Miyaoka}) on complex degree-$2h$ K3 surfaces is sharp for $h > 200$,
as claimed in Corollary \ref{cor}.
\end{rem}

\section{Complex case, odd $d$}
\label{s:odd_d}

In this section we use lattice enhancements and theory of lattice-polarized complex K3 surfaces to deal with the case when 
$d$ is odd.

\begin{prop}
\label{thm:odd_d_over_C}
For odd $d>1$ and every even $h\geq 2d(2d+3)-2$
there exists a complex projective degree-$2h$ $K3$ surface with exactly $24$ \emph{ smooth degree-$d$} rational curves. 
\end{prop}

\begin{proof}
To prove the proposition, we employ the same approach as in the previous section,
but we start from a specific K3 surface $W$ deforming inside a 5-dimensional family 
obtained by enhancing $X$ by a section $Q$ of height $2$ meeting all $I_2$ fibres in a different component than $O$,
to be called the non-identity component.
Explicitly, we let the divisor $D$  with $D^2=-8$
be the sum of generators of the four orthogonal $A_1$ summands in $T(X)$ in \eqref{eq:T(X)}.
Then $D/2\in T(X)^\vee$ glues to an element  $v\in\NS(X)^\vee$
determined by a given anti-isometry of discriminant groups $A_{T(X)}\cong A_{\NS(X)}$.
Presently we can set this up as follows.

$\NS(X)$ is an index 4 overlattice of $U+A_1^{12}$,
with generators of the $A_1$ summands denoted by $a_i\; (i=1,\hdots,12)$,
obtained by adjoining the vectors $v_1=(a_1+\hdots+a_8)/2$ and $v_2=(a_5+\hdots+a_{12})/2$
which correspond to two of the 2-torsion sections.
We develop a  basis of the discriminant group $A_{\NS(X)}$ which realizes the anti-isometry with $A_{T(X)} \cong A_{U(2)}^2 \oplus A_{A_1}^4$ by 
choosing 
\begin{itemize}
\item
$(a_1+a_2)/2, (a_1+a_3)/2$ corresponding to one copy of $A_{U(2)}$,
\item
$(a_5+a_6)/2, (a_5+a_7)/2$ corresponding to another copy of $A_{U(2)}$,
\item
$(a_4+a_8+a_i)/2$ for $i=9,\hdots,12$ for four copies of $A_{A_1}$.
\end{itemize}
It now follows that $D/2$ glues to $v=(a_9+\hdots+a_{12})/2$ in $\NS(W)$,
and $D/2+v+v_1$ is the class of the anticipated section $Q$
meeting each $I_2$ fibre at the non-identity component;
by construction, $D/2=\varphi(Q)$ is the image of $Q$ under the orthogonal projection $\varphi$
 from $\NS(W)_\QQ$
with respect to the trivial lattice $\Triv(W) = \Triv(X)$.

As before, we obtain a 5-dimensional family of complex K3 surfaces whose very general member $W$ 
has transcendental lattice
\[
T(W) = D^\perp = U(2)^2 \oplus A_3(2) \subset T(X).
\]
As a direct application of Criterion \ref{crit}, we obtain
\begin{lemm}
\label{lem:odd_d}
$H=NF+d(O+Q)$ is very ample for all $N>2d$.
\end{lemm}

Since $H^2=4Nd$, this implies Theorem \ref{thm:odd_d_over_C} for all $h>8d^2$ with $h\equiv 0\mod 2d$.
To cover all other residue classes modulo $2d$, we further enhance $\NS(W)$
by a divisor $D'\in T(W)$ with $D'^2<-4$.
Again, this results in a family of K3 surfaces, this time 4-dimensional.
By choosing $D'$ supported on $U(2)^2$, we can ensure that
the induced elliptic fibration on 
the very general member $V$ is  non-degenerate, i.e.\ it still has 12 $I_2$ fibres.

\begin{lemm}
\label{lem:odd_d_va}
$H+D'$ is very ample for all $N>(Q'.O)+3$.
\end{lemm}

\begin{proof}[Proof of Lemma \ref{lem:odd_d_va}]
The proof of the lemma proceeds as in the previous section by writing
$$
D' = \varphi(Q') = Q' - O - ((Q'.O)+2)F
$$
for some section $Q'\in\MW(V)$.
Then we check the conditions of Criterion \ref{crit}
as in Step 4, including the section $Q'$.
\end{proof}

With Lemma  \ref{lem:odd_d_va} at our disposal, Proposition  \ref{thm:odd_d_over_C} follows immediately
for all even $h\geq 2d(2d+3)-2$
by taking $D'$ with $D'^2=-8, -12,\hdots, -4d-4$
because then $-4d-4\leq D'^2 = \varphi(Q')^2 = -4-2(Q'.O)$, whence $(Q'.O)\leq 2d$.
%
%
\end{proof}

\begin{rem}
Since $T(W)$ only represents integers divisible by $4$,  the above approach cannot cover 
the  residue classes congruent to $2$  modulo $4$.
\end{rem}

\begin{summary}
Together with Proposition \ref{thm:even_d_over_C},
this proves Theorem \ref{thm} over algebraically closed fields of characteristic zero.
\end{summary}

\section{Algebraic approach}
\label{s:p>0}

To prove Theorem \ref{thm} in positive characteristic,
one can in principle pursue the same moduli-theoretic approach as in the previous sections.
However, there is the subtlety of proving that the resulting Noether--Lefschetz loci
parametrising the enhanced K3 surfaces $W$ resp.\ $V$ are non-empty.
In \cite{RS-23}, this kind of problem was overcome by considering terminal objects in higher Noether--Lefschetz strata.
Here we follow a different method by working, almost, with two single K3 surfaces doing the job directly for us,
thanks to their large Mordell--Weil lattices.
In fact, the arguments from Sections \ref{s:even_d}, \ref{s:odd_d}
carry over literally once we know that there is the desired section $Q$ (or also $Q'$),
and this can be achieved almost independently of the characteristic.
Our first surface will showcase this in a prototypical way.

\subsection{A special singular K3 surface}

Let $X$ denote the complex K3 surface with transcendental lattice $T(X)=$ diag$(12,12)$.
By \cite[Thm.\ 2.1]{Shioda-T}, following \cite{Kuwata}, $X$ can be given by the Weierstrass form
\begin{eqnarray}
\label{eq:X}
X: \;\;\; y^2 = x^3 - 3 t^4 x + t^{12} + 1
\end{eqnarray}
which exhibits $X$ as a triple cover of the Kummer surface of $E\times E$
where $E$ denotes the elliptic curve with an automorphism of order $4$.
The given elliptic fibration has 12 fibres of type $\IK_2$,
non-degenerate outside characteristics $2,3$,
Picard number $\rho(X)=20$ and Mordell--Weil group $\MW(X) \cong \ZZ^6\times (\ZZ/2\ZZ)^2$
over $\CC$.
The same holds in characteristic $p\equiv 1\mod 4$;
otherwise, i.e. if $p\equiv 3\mod 4, p>3$, the surface becomes supersingular and the Mordell--Weil rank is raised to $8$,
but in what follows we will only need the above subgroup which always persists;
for simplicity, we therefore state all results only for the complex K3 surface $X$.

\begin{lemm}
$\MWL(X) \cong A_1^-(1/2)^2 \oplus A_2^-(1/2)^2$.
\end{lemm}

\begin{proof}
Abstractly, this can be built up successively from quotients of $X$ by involutions.
In detail, the quotient by the involution $t\mapsto -t$ is a rational elliptic surface $S$ with 6 fibres of type $\IK_2$,
hence $\MWL(S)^-\cong (A_1^\vee)^2$ by \cite[\S 8]{SS-MWL} 
which pulls back to the two copies of $A_1^-(1/2)$ inside $\MWL(X)$.
Orthogonally to this, we find the pull-back of $\MWL(S')$ for $S'$ the quadratic twist of $S$ at $0, \infty$,
i.e. the minimal resolution of the quotient of $X$ by the Nikulin involution 
$$(x,y,t)\mapsto (x,-y,-t).
$$
Again, by \cite[Thm.\ 2.1]{Shioda-T}, this has $T(S')=$ diag$(6,6)$;
using this and the determinant formula \cite[Cor.\ 6.39]{SS-MWL}, one can prove that $\MWL(S') \cong (A_2^-)^2$,
again using involutions, this time based on the symmetry $t\mapsto 1/t$ of $\PP^1$ which respects
both  \eqref{eq:X} and the induced fibration on $S'$.
This pulls back to $A_2^-(2)^2\subset\MWL(X)$.
Another application of the determinant formula forces two extra divisibilities among the sections on $X$ thus obtained.
By inspection of the configuration of singular fibres and sections, and of the symmetries of $X$,
these divisibilities can only occur on the sections pulled back from $S'$, giving the claimed isometry.
\end{proof}

One can also make this fully explicit.
To this end, we translate the $\QQ$-rational 2-torsion section to $(0,0)$.
Then \eqref{eq:X} transforms to
\[
X: \;\;\; y^2 = x(x^2 - 3(t^4+1)x + 3(t^8+t^4+1)).
\]
This admits a height 4 section $P'$ with $x$-coordinate
$x(P') = 2\sqrt 3 t^2$ pulling back from $S'$,
which is 2-divisible by virtue of the section $\hat P$ with $x(\hat P)=\sqrt 3 (t^2+t+1)(t^2-\sqrt 3 t+1)$.
Applying the order $3$ automorphism $t\mapsto \zeta_3 t$ from \eqref{eq:X} to $\hat P$, we get $A_2^-(1/2)$,
supplemented by another orthogonal copy of $A_2^-(1/2)$,
gained by applying the order $4$ automorphism $\psi: t\mapsto \zeta_4t$.
Similarly, the section $Q_0$ with $x$-coordinate 
$x(Q_0)=(1 + 2\zeta_3)(t^2 - \zeta_3^2)(t^2 + \zeta_3)$ pulls back from $S$ 
and can be augmented by applying $\psi$ to give $A_1(1/2)^2$.

\begin{lemm}
The narrow Mordell--Weil lattice of $X$ equals
\[
\MWL^0(X) \cong A_1^-(2)^2 \oplus A_2^-(2)^2.
\]
\end{lemm}

\begin{proof}
By inspection of the singular fibres, any section $P\in\MW(X)$ satisfies $2P\in\MW^0(X)$,
hence 
$\MWL^0(X) \supseteq A_1^-(2)^2 \oplus A_2^-(2)^2$.
On the other hand, one can check explicitly with the above 6 sections and the 2-torsion sections
that there cannot be any further sections in $\MWL^0(X)$.
(The $8\times 12$ matrix indicating the fibres met at non-identity components has full rank over $\mathbb F_2$.)
\end{proof}

\begin{prop}
\label{prop:even_h}
For any $d>1$ and every even $h\geq 2d(2d+3)-2$, 
there exists a degree-$2h$  model of $X$ with exactly $24$ \emph{ smooth degree-$d$} rational curves
in any characteristic $p\neq 2,3$.
\end{prop}

\begin{proof}
As explained, it suffices to prove the claim over $\CC$.
For even $d$, this follows readily by applying the results from Section \ref{s:even_d}
because $\MWL^0(X)$ obviously represents any integer divisible by $4$ by the four square theorem.

For odd $d$, we first single out the height 2 section 
$Q=Q_0+\psi^*Q_0$ which exactly takes the shape required in Section \ref{s:odd_d}.
It can thus be augmented by any element in $\varphi(Q)^\perp\subset\MWL^0(X)$.
This lattice equals $A_1^-(4)\oplus A_2^-(2)^2$ and still represents all integers divisible by $4$
as an immediate application of the 290 theorem.
\end{proof}

\begin{summary}
Theorem \ref{thm} is proved for even $h\geq 2d(2d+3)-2$   in all characteristics $p\neq 2,3$.
\end{summary}

\begin{rem}
Systematically, $X$ can be found using Nishiyama's method \cite{Nishi}.
This amounts to exhibiting a rank 6 even negative-definite lattice lattice $L$
(of the same discriminant form as $T(X) =$ diag$(12,12)$)
embedding primitively into the Niemeier lattice $N(A_1^{24})$
such that the orthogonal complement $L^\perp$ has root sublattice $R(L^\perp)=A_1^{12}$.
\end{rem}

\subsection{Elliptic K3 surfaces  coming from Kummer surfaces}

To cover odd $h$ (and even $d$) in Theorem \ref{thm} in positive characteristic,
we start with Kummer surfaces of the Jacobians of genus 2 curves.
By \cite{Kumar}, these admit an elliptic fibration with two fibres of type $\IK_0^*$ and six $\IK_2$,
just like $S'$ above.
Hence quadratic base change ramified at the $\IK_0^*$ fibres yields a 3-dimensional family of K3 surfaces 
\begin{eqnarray}
\label{eq:Kumar}
y^2  &= & (x + 4 (a - 1) (b - 1) c (t^2 - a) (t^2 - b))\cdot\nonumber\\ 
&  & (x + 4 (b - 1) (c - 1) a (t^2 - b) (t^2 - c))\cdot\\
& & (x + 4 (c - 1) (a - 1) b (t^2 - c) (t^2 - a))\nonumber
\end{eqnarray}
with 12 fibres of type $\IK_2$ at the square roots of $a,b,c,ab,bc,ca$.
Very generally, one has $\MW\cong\ZZ^3\times(\ZZ/2\ZZ)^2$
with linearly independent sections  given by the following $x$-coordinates
\begin{eqnarray*}
P: & x(P) = 
4  (t^2 - a) (t^2 - b) (t^2 - c) (t^2-a b c)/t^2
&
h(P) = 2\\
Q: & x(Q) = 
-(4 a b c - 4 a b - 4 c^2 + 4 c) (t^2 - a) (t^2 - b)
&
h(Q)=1 \\
Q': & x(Q') =   
4  (1-c)((a b - a - b + a b/c )t^2 - a b (a + b - c - 1)) (t^2 - c) 
&
h(Q')=1
\end{eqnarray*}
More precisely, the sections are orthogonal with respect to the height pairing,
of height indicated by the last entry in each row.
Moreover, $P$ meets each $\IK_2$ fibre at the non-identity component,
and the same holds for $Q+Q'$ exactly as in the previous subsection.
It follows that the narrow Mordell--Weil lattice contains
\[
\MWL^0 \supseteq \langle 2Q, 2Q', P+Q+Q'\rangle \;\;\; \text{ with quadratic form } \;\;\; 
\begin{pmatrix} 4 & 0 & 2\\
0 & 4 & 2\\ 2 & 2 & 4
\end{pmatrix}
\]
(very generally as a finite index sublattice).
The quadratic form obviously only represents integers divisible by $4$
(as in the previous section),
so we specialize to a subfamily admitting a section $P'$ of height $3/2$.
Assuming $P'$ to take the shape 
\[
x(P') = d(t^2 - a)(t - \beta)(t - \gamma) \;\;\; \text{ where } \;\;\; \beta^2=b,\; \gamma^2=c,
\]
one easily finds that $d = -4 \beta \gamma (\gamma - 1) (\beta - 1) (a - 1)$
ensures $P'$ to meet the $\IK_2$ fibre at $-\beta\gamma$ at the non-identity component,
in addition to the fibres at $\pm \sqrt a, \beta, \gamma$.
Moreover  $a,\beta,\gamma$ have to lie on a hypersurface $Z\subset\mathbb A^3$ which is birational to the double cover of $\PP^2$ branched in the sextic
\[ 
( \gamma + 1) ( \beta + 1)( \beta +  \gamma)( \beta^2  \gamma +  \beta  \gamma^2 +  \beta^2 - 6  \beta  \gamma +  \gamma^2 +  \beta +  \gamma) = 0.
\]
One can show that this defines the singular K3 surface with transcendental lattice diag$(2,6)$,
but we will only need that one can work out explicit examples for $P'$ by using the rational curve in $Z$ given by $\beta+\gamma=-1$, for instance (on which the fibration \eqref{eq:Kumar} does not degenerate).
Using this, one can verify that $(P'.P)=2, (P'.Q)=0, (P'.Q')=1$ whence the height pairing returns 
a rank 4 sublattice
\[
 \langle 2Q, 2Q', P+Q+Q',2P'\rangle\subseteq\MWL^0
 \;\;\;
  \text{ with quadratic form } \;\;\; 
A = \begin{pmatrix} 4 & 0 & 2 & 2\\
0 & 4 & 2 & 0\\ 2 & 2 & 4 & 0\\
2 & 0 & 0 & 6
\end{pmatrix}.
\]
For a very general K3 surface $Y$ inside this 2-dimensional family, the above inclusion has finite index.

\begin{lemm}
The quadratic form $A/2$ represents every integer except for $1$. 
\end{lemm}

\begin{proof}
Direct check by 
applying \cite[Corollary 2]{Barowsky} to the exception $m=1$.
\end{proof}

By the same arguments as in the proof of Proposition \ref{prop:even_h}, we can thus prove the following.

\begin{prop}
\label{prop:even_d}
For any even $d$ and every odd $h\geq 2d(2d+3)-2$, 
there exists a degree-$2h$  model of $Y$ with exactly $24$ \emph{ smooth degree-$d$} rational curves
in any characteristic $p\neq 2,3$.
\end{prop}

\subsection{Proof of Theorem \ref{thm}}

Theorem \ref{thm} can be derived in any characteristic $\neq 2,3$ by combining Propositions \ref{prop:even_h} and \ref{prop:even_d}.
\qed

\begin{rem}
We emphasize that the above explicit argument also covers the characteristic zero case.
The alternative proof of that case in Sections \ref{s:even_d}, \ref{s:odd_d}
was given because it appears more conceptual and direct.
\end{rem}

\section{Concluding Remarks} \label{sec-finalremarks}

Theorem \ref{thm} leaves the case open where both $d$ and $h$ are odd.
Here we comment on this briefly, without proofs and only over $\CC$.

\subsection{}
Part of the open case can be covered by considering elliptic K3 surfaces with 8 fibres of type $\IK_3$
(that is, quadratic base changes of the Hesse pencil).
Combining \cite{degt}, \cite[\S 11.2]{RS-24} and the approach from Section \ref{s:even_d},
one can show:

\begin{prop}
For any $d\in\NN$ and  $h'\gg 0$,
there are complex elliptic K3 surfaces with 8 fibres of type $\IK_3$
admitting  model of degree $6h'+2d^2$ with smooth fibre components featuring as degree $d$ curves.
\end{prop}

\subsection{}
For degree $d>1$, 
elliptic K3 surfaces with 6 fibres of type $\IK_4$, as in \cite[\S 11.1]{RS-24}, 
do not yield anything additional compared to Theorem \ref{thm}.

\subsection{}
For  $h$ large compared to $d$, there are no K3 surfaces with exactly 23 smooth rational curves of degree $d$,
by the same arguments as in the line case in  \cite{degt}, thanks to the general bound from \cite{RS-24}.

\subsection{}
Constructing K3 surfaces containing 22 smooth rational curves (for large $h$)
is surprisingly subtle. In fact, for odd $d$ and $h$, we could not cover any additional values
compared with the examples extracted from \cite{degt}.

%
%
%

\subsection{}
\label{ss:2,3}

In characteristics $2, 3$, there are different bounds due to the occurrence of quasi-elliptic fibrations, see \cite{RS-24}.
For K3 surfaces of finite height, however, and for elliptic fibrations in general, Miyaoka's bound carries over,
but then some of the above constructions degenerate,
so the problem is decidedly more subtle.

\end{document}